\documentclass[journal,10pt,onecolumn,draftclsnofoot,]{IEEEtran}

\IEEEoverridecommandlockouts                              
\usepackage{graphicx}

\usepackage{amsmath, amsfonts, amssymb, mathrsfs, setspace, epsfig, amsthm,url,float,framed,circuitikz, color}
\usepackage[T1]{fontenc}
\usepackage[font=small,labelfont=bf,tableposition=top]{caption}
\usepackage[font=footnotesize]{subfig}
\usepackage[demo,abs]{overpic}
\usepackage[bookmarks=false]{hyperref}
\usepackage[ruled]{algorithm2e}
\newtheorem{theorem}{Theorem}
\newtheorem{definition}{Definition}
\newtheorem{lemma}{Lemma}

\newtheorem{problem}{Problem}

\newcommand{\bbb}{\mathbb}

\makeatletter
\usepackage[noadjust]{cite}
\makeatother

\def\eq#1{\begin{equation}#1\end{equation}}
\newcommand{\R}{{\rm I\!R}}
\def\rep#1{(\ref{#1})}
\def\scr#1{{\cal #1}}
\newcommand{\1}{\mathbf{1}}
\newcommand{\0}{\mathbf{0}}

\newcommand{\D}{\mathcal{D}}

\title{Multi-Competitive Viruses over Static and Time--Varying Networks}

\author{Philip E. Par\'{e}, Ji Liu,
Carolyn L. Beck,
Angelia Nedi\'{c}, and Tamer Ba\c{s}ar*\thanks{
* Philip E. Par\'{e}, Ji Liu, Carolyn L. Beck,  and Tamer Ba\c{s}ar are with the Coordinated Science Laboratory at the University of Illinois at Urbana-Champaign and can be reached at {\tt philip.e.pare@gmail.com}, {\tt jiliu@illinois.edu}, {\tt beck3@illinois.edu}, and {\tt basar1@illinois.edu}, respectively. Angelia Nedi\'{c} is with the School of Electrical, Computer and Energy Engineering at Arizona State University and can be reached at {\tt angelia.nedich@asu.edu}.  This material is based on research partially sponsored by the National Science Foundation, grants ECCS 15-09302, and CNS 15-44953.  All material in this paper represents the position of the authors and not necessarily that of NSF.
}}

\begin{document}
\maketitle

\begin{abstract}
Epidemic processes are used commonly for modeling and analysis of biological networks,  computer networks, and human contact networks. The idea of competing viruses has been explored recently, motivated by the spread of different ideas along different social networks. Previous studies of competitive viruses have focused only on two viruses and  on static graph structures. 
In this paper, we consider multiple competing viruses over static and dynamic
graph structures, and investigate the eradication and propagation of diseases in these systems. 
Stability analysis for the class of models we consider is performed and an antidote control technique is proposed. 
\end{abstract}

\section{Introduction}

Spread dynamics have been studied for hundreds of years. Bernoulli developed one of the first known models inspired by the smallpox virus \cite{bernoulli1760essai}.
In this paper we focus exclusively on   
susceptible-infected-susceptible (SIS) models, which have been developed for both continuous  \cite{kermack1932contributions,fall2007epidemiological,van2009virus,ahn2013global} and discrete  time domains \cite{ahn2013global,wang2003epidemic,peng2010epidemic}.  SIS models consist of a number of agents that are either infected or healthy (susceptible), which may cycle (aperiodically)  between these two states. 
The infection rate combined with the connectivity of the $i$th agent with infected neighbors $j$ (denoted by $\beta_{ij}$) positively affects the probability of being infected, while the healing rate $\delta_i$ negatively affects the infection probability. This is depicted in Figure \ref{fig:sis}.

The idea of two competing SIS viruses, namely the bi-virus model, has been recently pursued in \cite{prakash2012winner,wei2013competing,sahneh2014competitive,santos2015bivirus,liu2016onthe,arxiv}. 
The main motivation for such systems is that of competing ideas spreading on different social networks. However these models can have broader applications to political stances, adaptation of competing products, competing practices in farming, etc. and can be generalized to more than two viruses. Consider, for example, the case of three competing viruses; then each state has four possible states: susceptible, infected with virus $1$, $2$, or $3$. 
The idea of information diffusion on two layered networks has also been explored for a susceptible-infected-recovered (SIR) model in \cite{yagan2013conjoining}.  
In \cite{xu2012multi}, a different multi-virus model is considered. 


Further, all previous work on competing viruses has focused on viruses over static graph structures. There are recent results for the single virus model over time--varying networks \cite{prakash2010virus,bokharaie2010spread,rami2014switch,pare2015stability,pare2017epidemic}. Some of  the ideas from \cite{pare2015stability,pare2017epidemic} will be employed in this paper and applied to a more general model.

Various control techniques have been applied to SIS virus systems \cite{wan2007network,wan2008designing,vijayshankar2012cost,arxiv}. These techniques assume the healing rate is a control variable. 
In \cite{arxiv}, it is shown that there exists no distributed linear feedback control that can stabilize the system, and in fact, 
will destabilize the system. 
Alternative approaches focus on reducing the maximum eigenvalue of the linearized system using the healing rate and/or the infection rate. 
In \cite{wan2007network,wan2008designing}, distributed control techniques for setting healing rate and quarantine protocols are proposed and implemented on a severe acute respiratory syndrome (SARS) simulation model. 
In \cite{vijayshankar2012cost}, a bound is provided for the cost of fairness of mitigating the spread of disease, that is, the difference between the optimal solution and the fair or homogeneous solution, for several classes of graphs. In \cite{PreciadoTCNS14}, geometric programming ideas  are used to control single SIS virus systems and the authors present a polynomial time algorithm illustrated on  an air transportation network. 
In \cite{watkins2016optimal}, similar ideas to \cite{PreciadoTCNS14} are applied to the bi-virus model.

In this paper we present a generalization of the bi-virus model to an arbitrary number, $m$, of  competing viruses. We provide conditions for stability of the disease-free equilibrium (DFE) for static as well as time--varying graph structures. We also provide sufficient conditions for stability of the non-disease free equilibrium (NDFE). We provide two control techniques based on minimizing the maximum eigenvalue of the linearized system, appealing to some of the theorems presented herein. 
These control techniques, which are different from other approaches in the literature, allow every agent to have a base healing rate and an  additive control term. 

The paper is organized as follows: we first introduce, in Section \ref{sec:model}, the SIS model and the competing virus model 
for $m$ viruses. In Sections \ref{sec:analysis} and \ref{sec:ndfe} we analyze the model, providing conditions for stability of the  DFE and the NDFE, and in Section \ref{sec:control} we provide an antidote control formulation. In Section \ref{sec:sim}, we present a set of illuminating simulations of various competing virus models over time--varying networks, and we conclude with some discussion in Section \ref{sec:con}. 

\begin{figure}[!tbp]
    \centering
    \subfloat[SIS Model\label{fig:sis}]{%
      \begin{overpic}[width=0.24\textwidth]{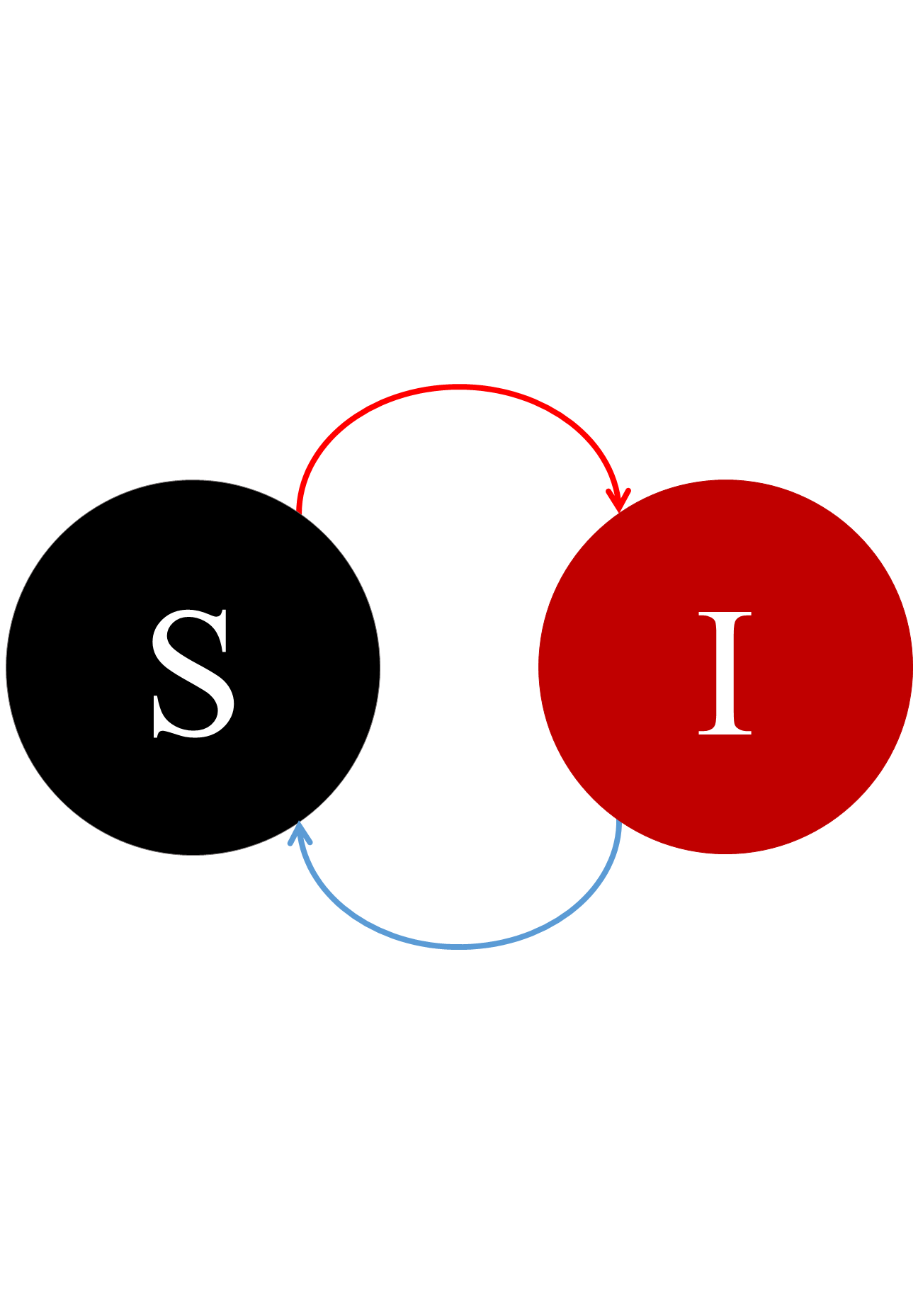}
        \put(46,51){{\parbox{0.05\linewidth}{%
     \begin{equation*}
          \delta_i
     \end{equation*}}}}
     \put(40,98.5){{\parbox{0.05\linewidth}{%
     $\sum \beta_{ij}  p_j$}}}
      \end{overpic}
    }%
    \subfloat[Model Three Competing Viruses\label{fig:sisis}]{%
      \begin{overpic}[width=0.309\textwidth]{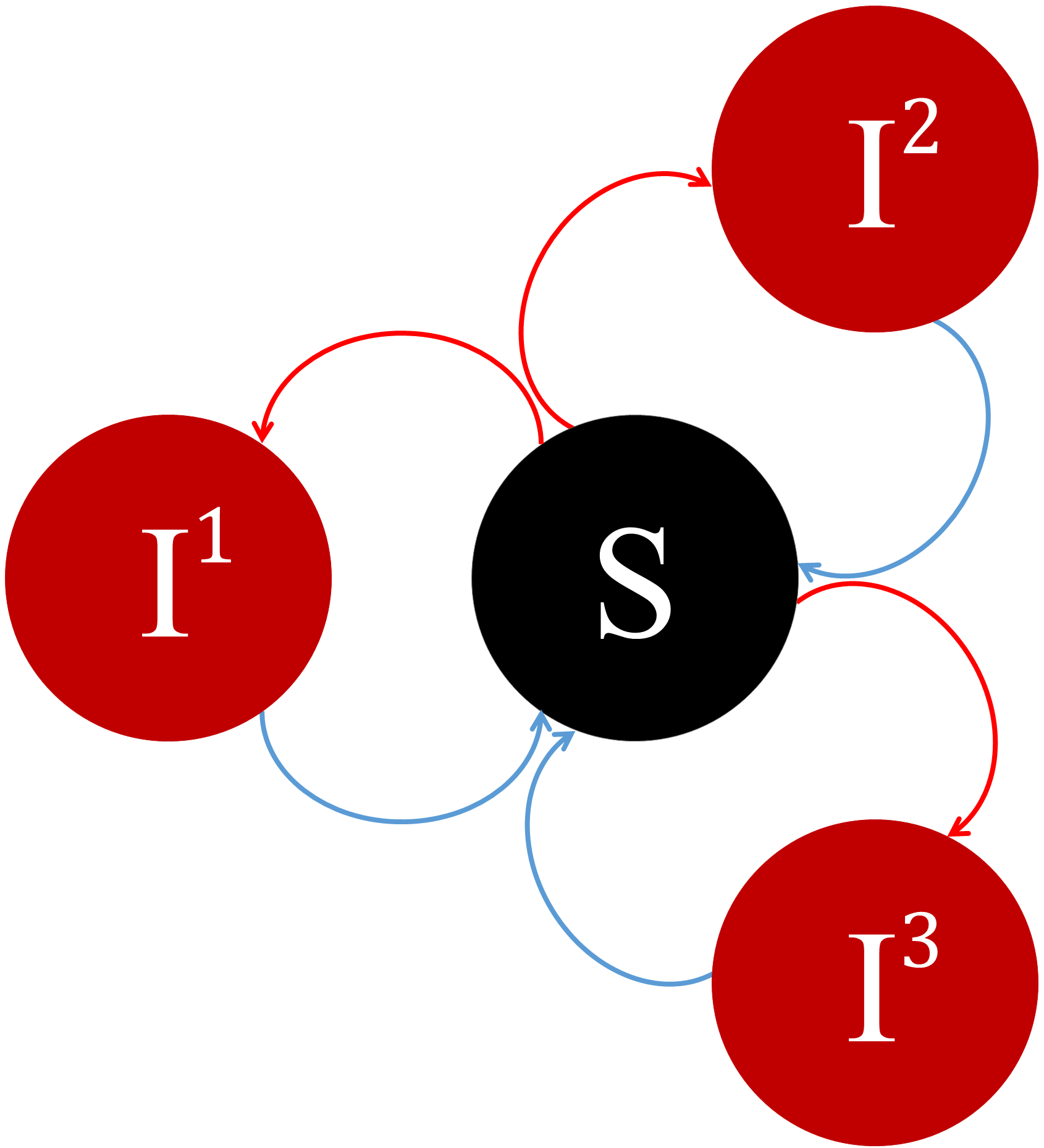}
        \put(46,51.5){{\parbox{0.05\linewidth}{%
     \begin{equation*}
          \delta^1_i
     \end{equation*}}}}
     \put(39.8,98.5){{\parbox{0.05\linewidth}{\small
     $\sum \beta^1_{ij}  p^1_j$}}}\normalsize
     \put(75,110){{\parbox{0.05\linewidth}{\small
     $\sum \beta^2_{ij}  p^2_j$}}}\normalsize
     \put(105,56){{\parbox{0.05\linewidth}{\small
     $\sum \beta^3_{ij}  p^3_j$}}}\normalsize
     \put(117,92){{\parbox{0.05\linewidth}{%
     \begin{equation*}
          \delta^2_i
     \end{equation*}}}}
     \put(70,36){{\parbox{0.05\linewidth}{%
     \begin{equation*}
          \delta^3_i
     \end{equation*}}}}
      \end{overpic}
    }
    \caption{The $i$th agent has the probability of being in either a susceptible or an infected state (the summations are over $j$). 
    }
\label{fig:models}
\end{figure}
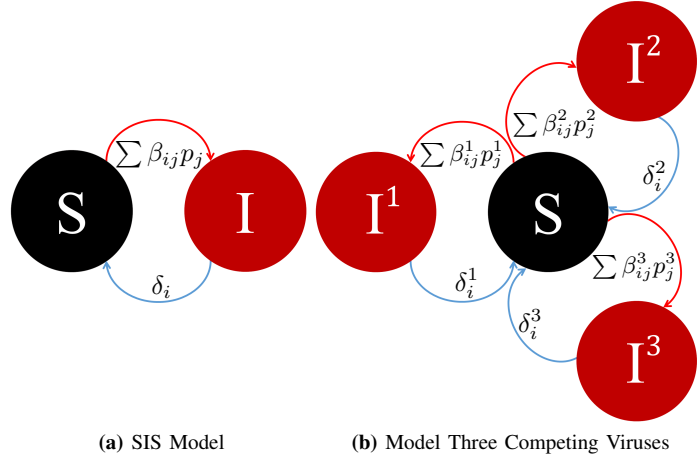

\subsection{Notation}

For any positive integer $n$, we use $[n]$ to denote the set $\{1,2,\ldots,n\}$.
We view vectors as column vectors. We use $x^T$ to denote the transpose of a vector $x$. 
The $i$th entry of a vector $x$ will be denoted by $x_i$.
The $ij$th entry of a matrix $A$ will be denoted by $a_{ij}$ and, also, by $[A]_{ij}$ when convenient.
We use $\0$ and $\1$ to denote the vectors whose entries are all equal to $0$ and $1$, respectively,
and $I$ to denote the identity matrix,
while the dimensions of the vectors and matrices are to be understood from the context.
For any vector $x\in\R^n$, we use ${\rm diag}(x)$ to denote the $n\times n$ diagonal matrix
whose $i$th diagonal entry equals $x_i$.
For any two sets $\scr{A}$ and $\scr{B}$,
we use $\scr{A}\setminus \scr{B}$ to denote the set of elements in $\scr{A}$ but not in $\scr{B}$.

For any two real vectors $a,b\in\R^n$, we write $a\geq b$ if
$a_{i}\geq b_{i}$ for all $i\in[n]$,
$a>b$ if $a\geq b$ and $a\neq b$, and $a \gg b$ if $a_{i}> b_{i}$ for all $i\in[n]$.
For a real square matrix $M$, we use 
$s(M)$ to denote the largest real part among its eigenvalues, i.e., 
$s(M) = \max \left\{{\rm Re}(\lambda)\ : \ \lambda\in\sigma(M)\right\},$
where ${\rm Re}(\cdot)$ is the real part of the argument and $\sigma(M)$ denotes the spectrum of $M$. For a symmetric matrix $M$, we use $\lambda_1(M)$ to denote its largest eigenvalue.


\section{The Model} \label{sec:model}
The generic SIS model, a generalization of the models introduced in \cite{van2009virus}, is 
\begin{equation}\label{eq:vani}
\dot{p}_i(t) = (1 - p_i(t))\sum_{j=1}^n \beta_{ij}  p_j(t) - \delta_i p_i(t),
\end{equation}
where $p_i$ is the probability that agent $i$ is infected, the $\beta_{ij}$'s are (possibly asymmetric) 
infection rates incorporating the nearest-neighbor graph structure, and $\delta_i$ is the healing rate. Neighbor relationships among the $n$ agents are described
by a directed graph $\bbb{G}$ on $n$ vertices with an arc from vertex $j$
to vertex $i$ whenever agent $i$ can be infected by agent $j$.
The agents can also be thought of as groups of people and $p_i$'s as the percentages of the groups that are infected, and therefore the neighbor graph $\bbb{G}$ can have self-arcs at all $n$ vertices.  
Hence, $\beta_{ij}$ equals zero if there is not an edge in $\bbb{G}$ from node $j$ to node $i$. The model in \eqref{eq:vani} is more general because the underlying graph $\bbb{G}$ can be directed and the weights given by $\beta_{ij}$ can be any non-negative number.
The representation in \eqref{eq:vani} can be put into matrix form:
\begin{equation}\label{eq:van}
\dot{p}(t) = (B- P(t)B - D)p(t),
\end{equation}
where $p$ is the vector of the $p_i$'s, $B$ is the matrix of the $\beta_{ij}$'s, $P = \text{diag}(p)$, and $D= \text{diag}(\delta_1,\dots,\delta_n)$. In the analysis that follows, as stated above, $B$ is not assumed to be symmetric unless explicitly stated so. 

This model has been extended to have two viruses, providing a generalization of the model introduced in  \cite{sahneh2014competitive}, 
\begin{equation}\label{eq:2sis}
\begin{split}
\dot{p}^{1}_i(t) &= (1 - p^{1}_i(t) - p^{2}_i(t))\sum_{j=1}^n \beta^{1}_{ij}  p^{1}_j(t) - \delta^{1}_i p^{1}_i(t), \\
\dot{p}^{2}_i(t) &= (1 - p^{1}_i(t) - p^{2}_i(t))\sum_{j=1}^n \beta^{2}_{ij}  p^{2}_j(t) - \delta^{2}_i p^{2}_i(t),
\end{split}
\end{equation}
where ${p}^{1}_i(t)$ and ${p}^{2}_i(t)$ are the probabilities that agent $i$ has virus $1$ and $2$ respectively, and each virus has its own infection rates and healing rates. Each virus spreads over a (possibly different) spanning subgraph of $\bbb{G}$, where their union is the neighbor graph $\bbb G$.
It will be assumed that both of the two subgraphs are strongly connected
and, thus, so is $\bbb{G}$.\footnote{
A directed graph is {\em strongly connected} if
for any two distinct vertices $i$ and $j$, there is a directed path from $i$ to $j$. 
}

We need not restrict ourselves to two viruses, however. A direct generalization leads to the following multi-virus model:
\begin{equation}\label{eq:sysi}
\dot{p}^{k}_i(t) = (1 - p^{1}_i(t) - \dots - p^{m}_i(t))\sum_{j=1}^n \beta^{k}_{ij}  p^{k}_j(t) - \delta^{k}_i p^{k}_i(t), 
\end{equation}
for all $k \in  [m]$. 
This representation can be written in matrix form as:
\begin{equation}\label{sys}
\dot{p}^{k}(t) = ((I - P^{1}(t) - \dots - P^{m}(t))B^{k} - D) p^{k}(t), 
\end{equation}
where the matrices are the same as in \eqref{eq:van}, but now they are dependent on which virus they correspond to. Since the subgraph for each virus $k$ is strongly connected, it follows  that $B^k$ is irreducible, meaning that it cannot be permuted into block triangular matrix form. The assumption that $B^k$ is bounded means $\forall i,j, \ \beta^{k}_{ij}<\infty$.

The set
\begin{equation}\label{D}
    \D     =\{(p^1,\dots , p^m) \; | \; p^k\ge \0, \ k\in [m], \; \sum_{k=1}^m p^k\le \1\}
\end{equation}
is invariant with respect to the system defined by \rep{sys}. If $p^{k}_i$ 
denotes the probability of agent $i$ being infected by virus $k$
and $1-\sum_{k=1}^m p_i^k$ denotes the probability of agent $i$ being healthy, it is natural to assume that their initial values are in $[0,1]$,
since otherwise the values will lack any physical meaning for the epidemic model considered herein. Similarly, if the states were representative of the density of infected members of a sub-population, they would also be bounded between zero and one.
\begin{lemma}
Suppose that for all $i\in[n],k\in[m]$, we have $\delta^{k}_i\ge0$, and the matrices $B^{k}$ are non-negative. 
If for all $i\in[n],k\in[m]$, we have $p^{k}_i(0),(1-p^{1}_i(0)-\cdots -p^{m}_i(0))\in[0,1]$, then  $p^{k}_i(t),p^{1}_i(t)+\cdots +p^{m}_i(t)\in[0,1]$ for all $i\in[n],k\in[m]$ and $t\ge 0$.
\label{box}
\end{lemma}
\begin{proof}
Suppose that at some time $\tau$, $p^{1}_i(\tau)+\cdots +p^{m}_i(\tau)\in[0,1]$ and $p^{k}_i(\tau)\in[0,1]$ for all $i\in[n],k\in[m]$.
Consider an index $i\in[n]$.
If $p^{k}_i(\tau)=0$, then from \rep{eq:sysi} and the assumption that the matrices $B^{k}$ are non-negative, $\dot p^{k}_i(\tau)\ge 0$.
The same holds for $p^{1}_i(\tau)+\cdots +p^{m}_i(\tau)$.
If $p^{k}_i(\tau)=1$, then from \rep{eq:sysi} and the assumption that the matrices $B^{k}$ are non-negative, $\dot p^{1}_i(\tau)\le 0$.
The same holds for $p^{1}_i(\tau)+\cdots +p^{m}_i(\tau)$.
It follows that $p^{k}_i(t),p^{1}_i(t)+\cdots +p^{m}_i(t)\in[0,1]$ for all $i\in[n],k\in[m]$ and $t\ge \tau$.

Since, by assumption, $p^{k}_i(0),(1-p^{1}_i(0)-\cdots -p^{m}_i(0))\in[0,1]$ for all $i\in[n],k\in[m]$,
it follows that $p^{k}_i(t),p^{1}_i(t)+\cdots +p^{m}_i(t)\in[0,1]$ for all $i\in[n],k\in[m]$ and $t\ge 0$.
\end{proof}
\noindent For the rest of the paper we assume $p^{k}_i(0),(1-\sum_j^mp^{j}_i(0))\in[0,1]$ for all $i\in[n],k\in[m]$.

It has been shown that there are disease-free equilibrium and non-disease free equilibria for the single virus system  \cite{VanMieghem2014pstar,khanafer2014stability,pare2015stability,pare2017epidemic}, as well as for  the two-virus system \cite{arxiv}; the same applies to multi-virus systems as well. However, in this case  the scenario becomes slightly more complicated because all viruses can reach the DFE, or a NDFE, or there may be some viruses at a DFE and some at a NDFE. We will explore several conditions for convergence to these different equilibria.

\section{Stability Analysis of the DFE} \label{sec:analysis}

First, we explore stability of the DFE for both the static and dynamic graph cases.

\subsection{Static Graph Structure}

We first give conditions under which the DFE is asymptotically stable. 
\begin{theorem}
Suppose that for all $i\in[n],k\in[m]$, we have $\delta^{k}_i\ge0$ and the matrices $B^{k}$ are non-negative and irreducible.
If $s(B^{k}-D^{k})\leq 0$ for all $k\in[m]$, then
the healthy state is the unique equilibrium of \rep{sys}, which is
asymptotically stable with  domain of attraction $\D$, as defined in \rep{D}.
\label{0global}\end{theorem}
\begin{proof}
To prove the theorem, it is sufficient to show that for all $k\in[m]$,
 $p^{k}(t)$ will asymptotically converge to $\0$ as $t\rightarrow\infty$
for any initial condition.

Since for all $k\in[m]$, $p^{k}_i(t)$ is always non-negative by Lemma \ref{box},
from \rep{eq:sysi}, 
\begin{eqnarray*}
\dot{p}^{1}_i(t) \le - \delta^{1}_i p^{k}_i(t) + (1 - p^{k}_i(t))\sum_{j=1}^n \beta^{k}_{ij}  p^{1}_j(t) ,
\end{eqnarray*}
which implies that the trajectories of $p^{k}_i(t)$ are bounded above by
a single-virus model. Since the $B^{k}$'s are non-negative and irreducible, by Proposition 3 in \cite{arxiv},
$p^{k}_i(t)$ will asymptotically converge to $\0$ as $t\rightarrow\infty$ for all $k\in[m]$,
and thus the healthy state is the unique equilibrium of \rep{sys}.
\end{proof}
We next state a result on global exponential stability for the case when  the underlying subgraphs are undirected and the infection rates are  symmetric.
\begin{theorem}\label{thm:1}
Suppose $B^k$ is symmetric, and the maximum eigenvalue of $B^{k}-D^{k}$ is less than zero, that is $ \lambda_1(B^{k}-D^{k})<0$. Then the DFE is exponentially stable for virus $k$, with  domain of attraction $\D$, 
in \rep{D}.
\end{theorem}
\begin{proof}
Consider the Lyapunov function $V(p^{k}) = 
\frac{1}{2}(p^{k})^T p^{k}$.  For $p^{k}\neq0$,
\begin{equation}
\begin{array}{lcl}
 \dot{V}(p^{k}) &=& (p^{k})^T\dot{p}^{k}\\ & = & (p^{k})^T(B^{k} -\sum_{l = 1}^m P^{l}B^{k} -D^{k})p^{k}  \\
 &\leq & (p^{k})^T(B^{k}-D^{k})p^{k} \\
 &\leq& \lambda_1(B^{k}-D^{k})\|p^{k}\|^2 <0.
\end{array}
\end{equation}
The first inequality holds because $(P^{l}B^{k})_{ij}\geq 0, \ \forall l,i,j$ by construction  since each $p^{l}_i(t)$ is a probability. The second inequality holds by the Rayleigh-Ritz Theorem  because $B^{k}-D^{k}$ is symmetric. Therefore, the system converges exponentially fast to the origin by Theorem 8.5 in \cite{khalil1996nonlinear}.
\end{proof}
\noindent Note that this is a generalization of the result in \cite{pare2015stability,pare2017epidemic}.

We can state that the condition in Theorem \ref{0global} is necessary and sufficient for eradication of all viruses.
\begin{theorem}\label{thm3}
Suppose  $\delta^{k}_i\ge0$, for all $i,k$, and the matrices $B^{k}$ are non-negative and irreducible for all $k$. 
The DFE (all $k$ viruses eradicated) is the unique equilibrium of \eqref{sys} if and only if
$s(B^{k}-D^{k})\leq 0$ for all $k \in  [m]$.
\end{theorem}
\begin{proof}
Sufficiency has been shown in Theorem \ref{0global}.
Therefore, to prove the theorem, all that needs to be shown is that if for any $j\in[m]$  $s(B^{j}-D^{j})> 0$,
the system \rep{sys} admits a NDFE.

Without loss of generality, suppose that $s(B^{1}-D^{1})> 0$. Set $p^{k}=\0$ for all $k=2,\dots,m$. Then,
the dynamics of $p^{1}$ simplifies to a single-virus system, which admits a NDFE
by Proposition 4 in \cite{arxiv}. Therefore, in the case when  $s(B^{1}-D^{1})> 0$,
the system \rep{sys} always admits an equilibrium of the form $(\tilde p^{1}, \0, \dots, \0)$ with $\tilde p^{1}\gg \0$.
\end{proof}


\subsection{Dynamic Graph Structure}

We can generalize the model from \eqref{eq:sysi} to have dynamic graph structure as
\begin{equation}\label{eq:ntv}
\dot{p}^{k}_i(t) = (1 - p^{1}_i(t) - \dots - p^{m}_i(t))\sum_{j=1}^n \beta^{k}_{ij}(t)  p^{k}_j(t) - \delta^{k}_i p^{k}_i(t), 
\end{equation}
where $\beta^{k}_{ij}(t)$ is a function of time and the equation holds for $k = 1,\dots,m$.
We now provide a sufficient condition for global exponential stability of the DFE.
\begin{theorem}\label{thm:1t}
Suppose $B^k(t)$ is symmetric, piecewise continuous in $t$, and bounded, and $\sup_{t\geq0}  \lambda_1(B^{k}(t)-D^{k})<0$. Then the DFE is exponentially stable for virus $k$, with  domain of attraction $\D$, 
in \rep{D}.
\end{theorem}
\begin{proof}
Consider the Lyapunov function $V(p^{k}) = 
\frac{1}{2}(p^{k})^T p^{k}$.  For $p^{k}\neq0$,
\begin{equation}
\begin{array}{lcl}
 \dot{V}(p^{k}) &=& (p^{k})^T\dot{p}^{k}\\ & = & (p^{k})^T(B^{k}(t) -\sum_{l = 1}^m P^{l}B^{k}(t) -D^{k})p^{k}  \\
 &\leq & (p^{k})^T(B^{k}(t)-D^{k})p^{k} \\
 &\leq& \lambda_1(B^{k}(t)-D^{k})\|p^{k}\|^2 \\
 &\leq& (\sup_{t\geq 0} \lambda_1(B^{k}(t)-D^{k}))\|p^k\|^2 <0.
\end{array}
\end{equation}
The first inequality holds because $(P^{l}B^{k}(t))_{ij}\geq 0, \ \forall l,i,j,t$ by construction  since each $p^{l}_i(t)$ is a probability. The second inequality holds by the Rayleigh-Ritz Theorem  because $B^{k}(t)-D^{k}$ is symmetric. The last inequality holds by definition of the supremum. Therefore, the system converges exponentially fast to the origin by Theorem 8.5 in \cite{khalil1996nonlinear}.
\end{proof}
\noindent This result is a generalization of Theorem 1 in \cite{pare2015stability,pare2017epidemic}.

We can also show exponential stability for the case when the infection rates are not symmetric and the underlying subgraphs are undirected, with some added assumptions.
\begin{definition}\label{def:gam}
For a given virus $k$, assume  that for all $t\geq 0$, there exist $c^{k}(t),\lambda^{k}(t)>0$ such that
\begin{equation}\label{eq:c}
\|B^{k}(t) - D^{k}\| \leq c^{k}(t) e^{-\lambda^{k}(t) t} \ \ \forall t\geq 0.
\end{equation}
We then define 
\begin{align}\label{eq:gam2}
 \gamma^{k}_1 &:= \sup_{t\geq0}\int_0^{\infty} c^{k}(t)^2 e^{-2\lambda^{k}(t) \tau}d\tau.
\end{align}
\end{definition}
Note that 
\begin{equation}\label{eq:gam3}
    \gamma^{k}_1 \geq \left\|\int_0^{\infty} e^{(B^{k}(t) - D^{k})^T\tau}e^{(B^{k}(t) - D^{k})\tau}d\tau\right\|.
\end{equation}
\begin{theorem}
Consider the dynamics for virus $k$ in \eqref{eq:ntv} with $B^{k}(t)$ continuously differentiable and $B^k(t) - D^k$ bounded, that is, there exists an $L>0$ such that $\|B^{k}(t)-D^{k}\|\leq L \ \forall t$. 
Assume that $\sup_{t\geq0} s(B^{k}(t)-D^{k})<0$ and $\gamma^{k}_1$ in Definition \ref{def:gam} is finite. 
If  $\sup_{t>0} \|\dot{B}^{k}(t) - D^{k}\| < \frac{1}{2(\gamma^{k}_1)^2}$ or $\int_t^{t+T}\|\dot{B}^{k}(s) - D^{k}\|ds \leq \mu T + \alpha$ for small enough $\mu>0$, then the DFE is exponentially stable for virus $k$, with  domain of attraction $\D$, 
in \rep{D}.
\end{theorem}
\begin{proof}
Note that since $(P^{l}(t)B^{k}(t))_{ij}\geq 0 \ \forall l,i,j$, by construction,
\begin{equation}
\begin{array}{lcl}
\dot{p}^k &=& (B^{k}(t) -\sum_{l = 1}^m P^{l}B^{k} -D^{k})p^k  \\
 &\leq & (B^{k}(t) - D^{k})p^k.
\end{array}
\end{equation}
Therefore, by Gr\"{o}nwall's Inequality (\cite{khalil1996nonlinear}), the solution of the original system will be bounded above by the solution of the linear system. 
Thus by Lemma 2 in \cite{pare2017epidemic}, the DFE is exponentially stable for virus $k$.
\end{proof}
\noindent Note that this theorem is a generalization of a single virus result provided in \cite{pare2017epidemic}, where Lemma 2 in \cite{pare2017epidemic} is for a less general model; however the same arguments hold by replacing $BA(t)$ with $B^k(t)$ and $B\dot{A}(t)$ with $\dot{B}^k(t)$.

\begin{theorem}\label{thm:nonH}
Consider the dynamics for virus $k$: 
\begin{equation*}
  \dot{p}^k  = (B^k(t)+\Delta^k(t) - P(t)(B^k(t)+\Delta^k(t)) - D^k )p^k.
\end{equation*}
Assume that
\begin{equation}\label{eq:a1} 
\lim_{T\rightarrow \infty}\frac{1}{T} \int_{t_0}^{t_0+T} \| B^k(s) - D^k\| ds \leq a < \infty,
\end{equation}
for all $t_0\geq 0$, and for some $\nu>0$ there exists an $h>0$ such that 
\begin{equation}\label{eq:a2}
\|B^k(t + h) -B^k(t)\| \leq \nu h^{\gamma},
\end{equation}
for all $t\geq 0$ and some $\gamma,$ $0 < \gamma \leq 1$. Assume further that
\begin{equation}\label{eq:avg}
 \lim_{T\rightarrow \infty}\frac{1}{T} \int_{t_0}^{t_0+T} s_1( B^k(s) - D^k)ds \leq \bar{\alpha},
\end{equation} 
for some negative scalar $\bar{\alpha}$ and for all $t_0\geq 0$,  
\begin{equation}\label{eq:a4}
 \lim_{T\rightarrow \infty}\frac{1}{T} \int_{t_0}^{t_0+T} \| \Delta^k(s)\| ds \leq \eta < \infty,
\end{equation} 
for all $t_0\geq 0$, and for all $i,j$ and $t\geq0$ the perturbation
\begin{equation}\label{eq:per}
    |\Delta_{ij}^k(t)| \leq \beta^k_{ij}(t).
\end{equation}
Then the origin is exponentially stable for virus $k$.
\end{theorem}
\begin{proof}
Since $(P(t)(B^k(t)+\Delta^k(t)))_{ij}\geq 0 \ \forall i,j$ by \eqref{eq:per} and Lemma \ref{box}, 
\begin{align*}
\dot{p}^k &= (B^k(t)+\Delta^k(t) - P(t)(B^k(t)+\Delta^k(t)) - D^k )p^k  \\
 &\leq  (B^k(t)+\Delta^k(t)) - D^k)p^k.
\end{align*}
Therefore, by Gr\"{o}nwall's Inequality (\cite{khalil1996nonlinear}), the solution of the original system will be bounded above by the solution of the linear system. 
Thus by Lemma 2 in \cite{pare2017epidemic}, the origin is exponentially stable for virus $k$.
\end{proof}
This result says that if the linearized system is Hurwitz on the average (not strictly Hurwitz for all time, as in the other theorems up to this point), then the system converges to the DFE. 
This fact is useful in the control design in Section~\ref{sec:control}.

\section{NDFE for the Multi-Virus Case}\label{sec:ndfe}

There are a number of different epidemic equilibria. The simplest scenario is when one virus is in an epidemic state and the remaining viruses are eradicated. 
\begin{theorem}
Assume $\delta^{k}_i\ge0$, for all $i,k$, and the matrices $B^{k}$ are non-negative and irreducible for all $k$.
If for some $i\in [m]$, $s(B^{i}-D^{i})> 0$ and $s(B^{k}-D^{k})\leq 0$ for all $k\neq i$, then
\eqref{sys} has two equilibria,
the healthy state $(\0,\dots , \0)$, which is asymptotically stable with domain of attraction $\{(p^1,\dots,p^m )| p^i = \0 \text{ and } p^k \in [0,1]^n \ \forall k \neq i\}$, and a unique epidemic state of the form $(\0, \dots, \0, \tilde p^{i}, \0, \dots, \0)$ with $\tilde p^{i}\gg \0$, which is
asymptotically stable with  domain of attraction $\D\setminus\{(p^1,\dots,p^m )| p^i = \0 \text{ and } p^k \in [0,1]^n \ \forall k \neq i\}$, with $\D$ defined in \eqref{D}.
\label{eglobal}\end{theorem}
\noindent Note that this result is an extension of Theorem 3 in \cite{arxiv}, and 
we present a sketch of the proof. 

{\em Sketched proof of Theorem \ref{eglobal}:}
From the proof of Theorem~\ref{0global}, $p^{k}(t)$ will asymptotically converge to $\0$
as $t\rightarrow\infty$ for all initial values $(p^{1}(0),\dots p^m(0)) \in \{(p^1,\dots,p^m )| p^i = \0 \text{ and } p^k \in [0,1]^n \ \forall k \neq i\}$, for $k\neq i$.
From \rep{sys},
$$\dot{p}^{i}(t) = ( B^{i}- D^{i} - P^{i}(t)B^{i}) p^{i}(t) - \sum_{k\neq i} P^{k}(t)B^{k}p^{k}(t).$$
Thus, we can regard the dynamics of $p^{i}(t)$ as an autonomous system
\eq{\dot{p}^{i}(t) = (B^{i}- D^{i} - P^{i}(t)B^{i}) p^{i}(t),\label{temp}}
with a vanishing perturbation $- \sum_{k\neq i} P^{k}(t)B^{i}p^{i}(t)$, which converges to $\0$  as $t\rightarrow\infty$.
From Proposition 5 in \cite{arxiv}, the autonomous system \rep{temp} will asymptotically converge to
a unique epidemic state $(\0, \dots, \0, \tilde p^{i}, \0, \dots, \0)$ for any $(p^{1}(0),\dots,p^{m}(0))\in\D\setminus\{(p^1,\dots,p^m )| p^i = \0 \text{ and } p^k \in [0,1]^n \ \forall k \neq i\}$, with $\D$ defined in \rep{D}. 
\hfill
$\qed$

Another possible NDFE is that of coexisting equilibrium, which is where more than one virus survives. We have the following interesting result similar to Theorem 7 in \cite{arxiv}. 

\begin{theorem}
Consider the model in \rep{eq:sysi} with each virus propagating over the same strongly connected graph $\bbb{G}$ with the corresponding adjacency matrix $A$, and each virus homogeneous in healing and infection rates, that is, for each $k \in [m]$  $\delta^{k}_i=\delta^{k}>0 \ \forall i\in[n]$
and $\beta^{k}_{i}=\beta^{k}>0 \ \forall i\in[n]$.  Suppose that $s(A)>\frac{\delta^{1}}{\beta^{1}} = \cdots = \frac{\delta^{m}}{\beta^{m}}$.
If $(\tilde p^{1},\dots,  \tilde p^{m})$ with $\tilde p^{k}>\0  \ \forall k\in[m]$ is an equilibrium 
of \rep{eq:sysi}, then $\tilde p^{k}\gg \0  \ \forall k\in[m]$ and 
$\tilde p^{i} = \alpha^{ik}\tilde p^{k} \ \forall i,k\in[m]$, for some constant $\alpha^{ik}>0$.
\label{parallel1}\end{theorem}

{\em Proof:} 
The homogeneity assumption on the infection rates allows us to factor $B^k = \beta^k A$, for each virus $k\in [m]$. To be an equilibrium of \rep{eq:sysi} the following must hold for all $k \in [m]$ 
\begin{align}\label{equil}
\begin{split}
(I-\tilde P^{1}-\cdots -\tilde P^{m})A\tilde p^{k} = \frac{\delta^{k}}{\beta^{k}}\tilde p^{k},
\end{split}
\end{align}
in which $(I-\tilde P^{1}-\cdots -\tilde P^{m})A$ is an irreducible Metzler matrix\footnote{
A matrix is Metzler if all of the off-diagonal components are non-negative.
}, since $\tilde P^{1}+\cdots +\tilde P^{m}$ is diagonal and $[\tilde P^{1}+\cdots +\tilde P^{m}]_{ii}<1$ for all $i\in [n]$, by Lemma \ref{box}. 
From Lemma 2 in \cite{arxiv}, it must be true that $\tilde p^{k}\gg \0  \ \forall k\in[m]$ and 
$\tilde p^{i} = \alpha^{ik}\tilde p^{k} \ \forall i,k\in[m]$, for some constant $\alpha^{ik}>0$.
\hfill
$\qed$

While the stability of the time--varying case has been explored in \cite{pare2017epidemic}, the time--varying NDFE or epidemic limit cycle 
is an open problem, even for the single virus case. Some work has been done to show the existence of a periodic NDFE for a single virus switching system in \cite{rami2014switch}.


\section{Antidote Control Formulation} \label{sec:control}


Let us assume that for each agent, in addition to the healing rate, there is a control input  $u_i(t)$ that acts as an additive boost to the healing rate. This implies that the controller can increase the agents' ability to recover from the virus, which can be thought of as the  administration of an antidote or some other type of treatment. This effect is portrayed in the model as 
\begin{equation*}
    \dot{p}^{k}_i(t) = (1 - p^{1}_i(t) - \dots - p^{m}_i(t))\sum_{j=1}^n \beta^{k}_{ij}  p^{k}_j(t) - (\delta^{k}_i + u_i^k(t))p^{k}_i(t).
\end{equation*}
We define $U(t) = \text{diag}(u)$ with $u = [u_1(t),\dots,u_n(t)]^T$. 
To simplify the discussion in this section, we assume that $B^{k}(t)$ is symmetric,  piecewise continuous in $t$, and bounded $\forall t \geq 0$. Similar to the approaches in \cite{wan2007network,wan2008designing,vijayshankar2012cost,PreciadoTCNS14,bullo2014control}, we focus on minimizing the maximum eigenvalue of $B^{k}(t)-(D^{k} +U^k(t))$. 
Even though these control techniques are generally effective, we believe the approaches herein are more general and simpler, and therefore more scalable. Also, the assumption that our control input is additive to the base healing rate is novel and more sensible for the main motivating example, that is, every agent should have some inherent healing rate that should not be affected by the controller.

While the solutions to the following posed problems may not meet the conditions of Theorems \ref{thm:1} and \ref{thm:1t}, that is, they may not result in the maximum eigenvalues being less than zero, they push the system towards those conditions, consistent with the principle of the average being less than zero, presented in Theorem \ref{thm:nonH}. And in practice, illustrated by simulation in the next section, these techniques reduce the spread of the epidemics. 
Under the aforementioned assumptions we can formulate the following optimization problem for each virus $k$, appealing to Theorems \ref{thm:1} and \ref{thm:1t} depending on whether $B^{k}$ is constant or time dependent:
\begin{equation*}
\begin{aligned}
& \underset{u^{k}_i(t)}{\text{minimize}}
& & \lambda_1(B^{k}(t)-(D^{k} +U^k(t)))\\
& \text{subject to}
 & &\sum_{i=1}^n u^{k}_i(t)\leq c^{k}, \  \; t\geq0,\\ 
 & & & U^k(t) = \text{diag}(u^{k}_1(t), \dots , u^{k}_n(t)), \\ 
 & & & u^{k}_i(t)\geq 0, \  \; i = 1, \ldots, n, \ t\geq0.\\
\end{aligned}
\end{equation*}
\noindent given that $B^{k}(t)$ is symmetric for all $t\geq 0$. 

From the Gershgorin Disc Theorem \cite{horn2012matrix} it is clear that by sufficiently increasing the $u^{k}_i$'s, 
the conditions of Theorems  \ref{thm:1} and \ref{thm:1t} will be satisfied. Therefore we can relax the above optimization problem to obtain the following:
\begin{problem}
\begin{equation*}
\begin{aligned}
& \underset{\eta^{k},u^{k}_i(t)}{\text{minimize}}
& & \eta^{k} \\
& \text{subject to}& &\eta^{k} \geq  \sum_{j = 1}^n \beta^{k}_{ij}(t) - (\delta^{k}_i + u_i^k(t)), \\
& & &\sum_{i=1}^n u^{k}_i(t)\leq c^{k}, \\
& & & u^{k}_i(t)\geq 0, \  \; i = 1, \ldots, n, \ t\geq0.\\
\end{aligned}
\end{equation*}\label{prob:1}
\end{problem}
\noindent This is clearly a linear program and can  easily be solved.

To make this a more compelling and realistic problem, we can impose a constraint on the number of agents that can be affected, which is a reasonable assumption because 
the cost of providing a low-dose treatment to all agents is higher than providing that same treatment dose to a few select members of the population (such as the sickest or most susceptible agents).  
Define the sparsity metric $\| \cdot \|_0$ as the number of the non-zero entries in its argument. 

Employing the sparsity metric, we have the following problem, 
with a capacity constraint and a sparsity constraint:
\begin{equation*}
\begin{aligned}
& \underset{\eta^{k},u^{k}_i(t)}{\text{minimize}}
& & \eta^{k} \\
& \text{subject to}& &\eta^{k} \geq \sum_{j = 1}^n \beta^{k}_{ij}(t) - (d^{k}_i+ u_i^k(t)), \\
& & &\sum_{i=1}^n u^{k}_i(t)\leq c^{k}, \\
& & &\| u^{k}(t)\|_0 \leq d^{k}, \\
& & & u^{k}_i(t)\geq 0, \  \; i = 1, \ldots, n, \ t\geq0,
\end{aligned}
\end{equation*}
\noindent where $d^{k}$ is the maximum number of agents that can be treated for virus $k$. 
At first glance, the second and third  constraints may seem redundant; however, the $\ell_1$ constraint limits the total amount of antidote that can be used while the sparsity constraint limits the number of agents that can be treated. The inclusion of the $\ell_1$ constraint prevents an infinite amount of antidote being administered to the limited number of agents allowed by the sparsity constraint.

It is well known that $\| \cdot \|_0$ is highly non-convex \cite{wright2009robust}, making the above problem difficult to solve. 
Therefore, to solve it we employ another relaxation using the reweighted $\ell_1$ norm \cite{candes2008enhancing}.
\begin{definition}
The weighted $\ell_1$ norm is 
\begin{equation}\label{eq:wl1}
\|x^k \|_{\hat{\ell}_1} : = \sum_{i=1}^n w_i^k |x^k_i|,
\end{equation}
where $w_i$'s are positive and can be a constant or depend on time. 
\end{definition}
\noindent In view of this,  we can rewrite the above problem  as the following:
\begin{problem}
\begin{equation*}
\begin{aligned}
& \underset{\eta^{k},u^{k}_i(t)}{\text{minimize}}
& & \eta^{k} + \kappa \| u^{k}(t)\|_{\hat{\ell}_1} \\
& \text{subject to}& &\eta^{k} \geq  \sum_{j = 1}^n \beta^{k}_{ij}(t) - (\delta^{k}_i + u_i^k(t)), \\
& & &\sum_{i=1}^n u^{k}_i(t)\leq c^{k}, \\
& & & u^{k}_i(t)\geq 0, \  \; i = 1, \ldots, n, \ t\geq0,
\end{aligned}
\end{equation*}\label{prob:2}
\end{problem}
\noindent where $\kappa$ is a constant weighting factor. 

An effective heuristic for the selection of the weights $w_i^k$'s in \eqref{eq:wl1}, proposed in \cite{candes2008enhancing}, is, for some small $\epsilon >0$, 
\begin{equation}
w_i^{k+1} = \frac{1}{|x^k_i|+ \epsilon}.
\end{equation}
For completeness, we include Algorithm \ref{alg:1}, which explains the implementation of this heuristic to solve Problem \ref{prob:2}. The notation Problem \ref{prob:2}$(w_{k-1})$ indicates that $w_{k-1}$ is used for the weighted $\ell_1$ norm in the objective function of Problem \ref{prob:2} in the $k$th iteration. Employing this heuristic yields a good solution to Problem \ref{prob:2} but clearly  is expensive, since it requires the calculation of multiple solutions. The effectiveness of this approach is illustrated in the following section via simulation.

\begin{algorithm}
   $w^0 = vec(\frac{1}{n}, \dots, \frac{1}{n})$\;
   $k=1$\;
   \While{$\|u_k-u_{k-1}\| > \varepsilon$}{
      $u_{k} = \text{arg}\min$ Problem \ref{prob:2}$(w_{k-1})$\;
      $w_i^{k} = \frac{1}{|u^k_i|+ \epsilon}$\;
      $k = k+1$\;
      }
    \caption{Algorithm for solving Problem \ref{prob:2}}\label{alg:1}
 \end{algorithm}



\section{Simulations}\label{sec:sim}

\begin{figure}
    \centering
    \subfloat[The system at time zero.\label{fig:thm3_0}]{
      \includegraphics[width=.49\columnwidth]{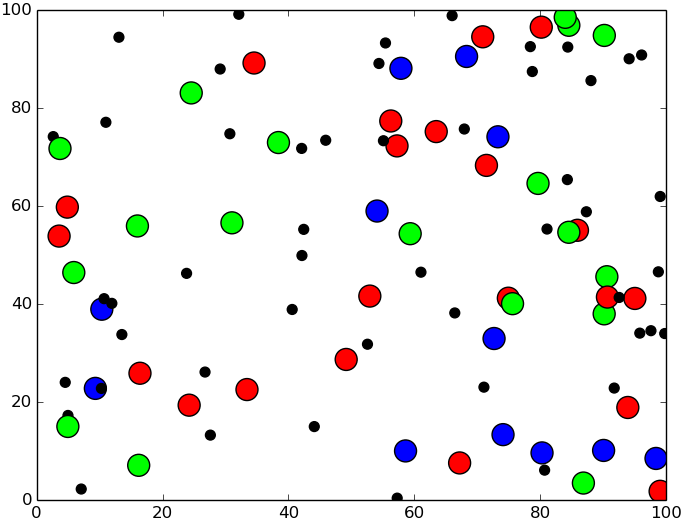}}
    \hfill
    \subfloat[The system at time 1000.\label{fig:thm3_400}]{
      \includegraphics[width=.49\columnwidth]{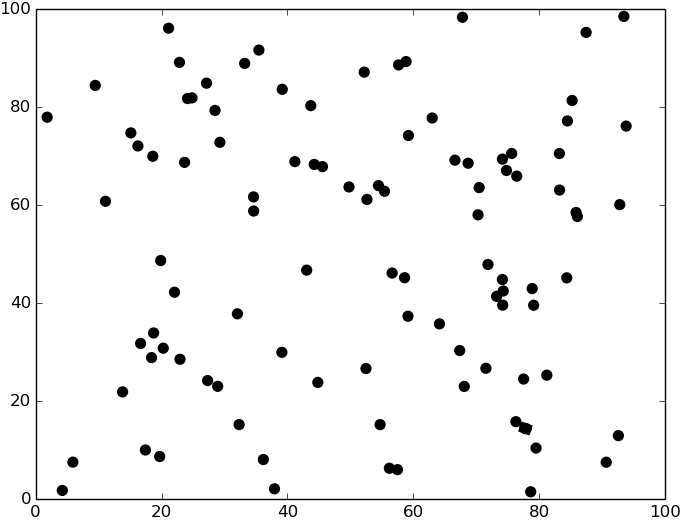}}
    \caption{This three-virus system meets the assumptions of Theorem \ref{eglobal} so virus 1 reaches an NDFE while the other two viruses die out.  
    The colors and diameters follow \eqref{eq:color} and \eqref{eq:diam} and the graph structure follows \eqref{eq:at}-\eqref{eq:phipiece}. 
A video of this simulation can be found at  \href{https://youtu.be/j_MHm08dA_o}{youtu.be/j\_MHm08dA\_o}.}
\label{fig:thm3}
\end{figure}

In this section we present a set of illuminating simulations of various competing virus models over static and time--varying graph structure networks. 
Due to limit of dimensions in color and size, for the simulations we will only have three competing viruses. 
Virus 1 is depicted by the color red ($r$), virus 2 is depicted by the color blue ($b$), and virus 3 is depicted by the color green ($g$). For all $i\in[n]$, the color at each time $t$ for agent $i$ is given by 
\begin{equation}\label{eq:color}
\frac{p^{1}_i(t)}{\sum_{k=1}^3 p^{k}_i(t)}r + \frac{p^{2}_i(t)}{\sum_{k=1}^3 p^{k}_i(t)}b + \frac{p^{3}_i(t)}{\sum_{k=1}^3 p^{k}_i(t)}g.
\end{equation}
When $p^{1}_i(t)+p^{2}_i(t)+p^{3}_i(t)=0$, the color goes to black, indicating completely healthy, susceptible.
These are used to facilitate the depiction of the parallel equilibrium ($\tilde p^{1} = \alpha^2\tilde p^{2} = \alpha^3\tilde p^{3}$), which will be shown by all nodes converging to the same color.
For all $i\in[n]$, the diameter of the node representing agent $i$ is given by 
\begin{equation}\label{eq:diam}
d_0 + (p^{1}_i(t)+p^{2}_i(t)+p^{3}_i(t))r_0,
\end{equation}
with $d_0$ being the default/smallest diameter and $r_0$ being the scaling factor depending on the total sickness of agent $i$. 
Therefore the color indicates the sickness each agent has and the diameter indicates 
how sick each agent is. 

For systems that have three different subgraphs, viruses 1, 2, and 3 spread on the graphs depicted by gray, green, and pink  edges, respectively. 
If all viruses spread on the same graph, the edges are gray.

\begin{figure}
    \centering
    \includegraphics[width = .49\columnwidth]{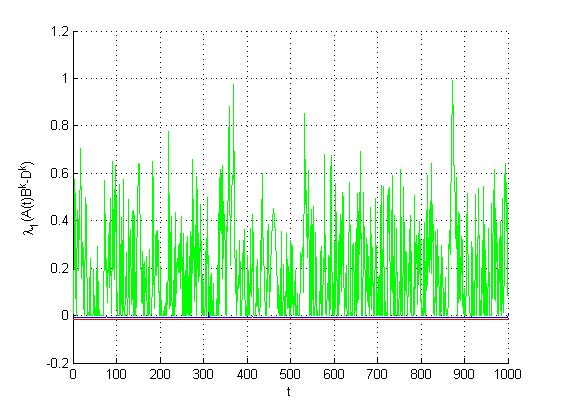}
    \caption{The maximum eigenvalues of the three viruses from the simulation in Figure \ref{fig:thm3}.}\label{fig:eig}
\end{figure}
The simulation in Figure \ref{fig:thm3} has three viruses spreading over the same time--varying graph. Similar to \cite{pare2017epidemic}, the graph structure is determined by 
\begin{equation}\label{eq:at}
\beta_{ij}(t) = \begin{cases}
    \beta e^{-\|z_i(t) - z_j(t)\|^2}, & \text{if } \|z_i(t) - z_j(t)\| < \hat{r}, \\
    0,              & \text{otherwise},
\end{cases}
\end{equation}
where $z_i(t)\in \mathbb{R}^2$ is the position of agent $i$, with $\hat{r}=10$. The agents have piece-wise constant drifts, that is,
\begin{equation}\label{eq:phi}
\dot{z}(t) = \phi(t),
\end{equation}
where $\phi(t)\in \mathbb{R}^2$ and is determined, for each dimension $l\in [2]$, by
\begin{equation}\label{eq:phipiece}
\phi_l =\begin{cases}
    -\phi_l, & \text{if } z_l=z_{c_l}+\gamma/2 \text{ or } z_l=z_{c_l}-\gamma /2 \\
   \ \ \phi_l ,             & \text{otherwise},
\end{cases}
\end{equation}
where the agents hover around a square, centered at some point $z_c$. The initial positions and $\phi$'s are chosen randomly. Each virus is homogeneous in infection rate. 
The first two viruses meet the assumptions of Theorem \ref{thm:1t}, while the maximum eigenvalue of the third virus fluctuates between  being positive and negative. See Figure \ref{fig:eig} for a plot of the maximum eigenvalues of the three-virus dynamics.  Consistent with the theorem, the first two viruses are eradicated quite quickly. The third virus is also eliminated, but it takes a little longer. This eradication 
is illustrated in Figure \ref{fig:thm3_400}. 


The simulation in Figure \ref{fig:dom} meets the assumptions of Theorem \ref{eglobal}, where $s(B^{1}-D^{1})> 0$, and $s(B^{2}-D^{2})< 0$ and $s(B^{3}-D^{3})< 0$. Therefore the first virus, depicted in red, reaches an epidemic equilibrium, while the other two viruses are eradicated.
\begin{figure}
    \centering
    \subfloat[The system at time zero.\label{fig:dom0}] {
      \includegraphics[width=.485\columnwidth]{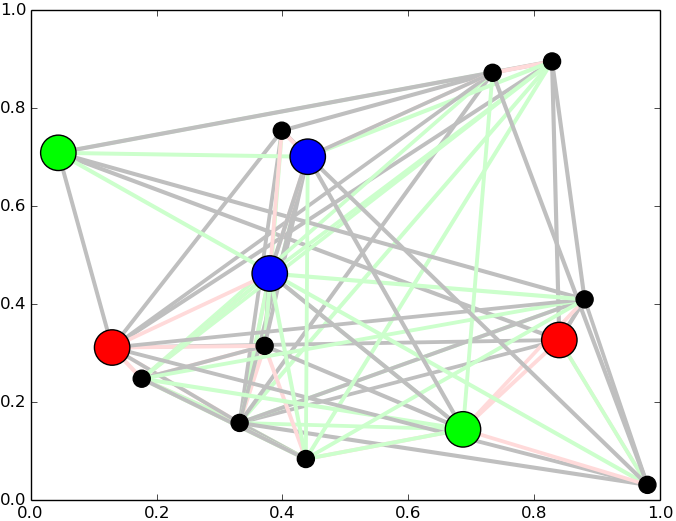}}
    \hfill
    \subfloat[The system at time 300.\label{fig:dom400}] {
      \includegraphics[width=.485\columnwidth]{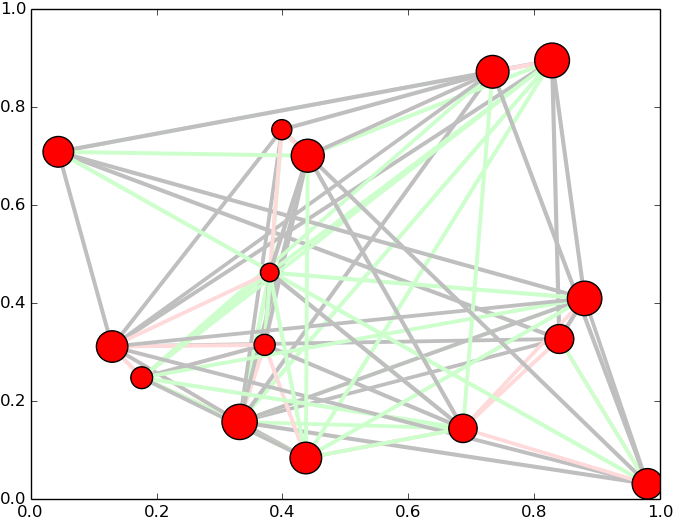}}
    \caption{This three-virus system meets the assumptions of Theorem \ref{eglobal} so virus 1 reaches an NDFE while the other viruses are eradicated.  
    The colors and diameters follow \eqref{eq:color} and \eqref{eq:diam}. 
A video of this simulation can be found at  \href{https://youtu.be/zCRiLr8sWEM}{youtu.be/zCRiLr8sWEM}.}
\label{fig:dom}
\end{figure}

The simulation shown in Figure \ref{fig:par} meets the assumptions of Theorem \ref{parallel1}, that is, the three viruses are each homogeneous, with $\frac{\delta^1}{\beta^1}=\frac{\delta^2}{\beta^2}=\frac{\delta^3}{\beta^3}$, and propagate over the same graph structure. There are $15$ agents and the initial conditions are given in Figure \ref{fig:par0}. Consistent with the theorem, the system converges to a co-existing parallel equilibria.

\begin{figure}
    \centering
    \subfloat[The system at time zero.\label{fig:par0}] {
      \includegraphics[width=.485\columnwidth]{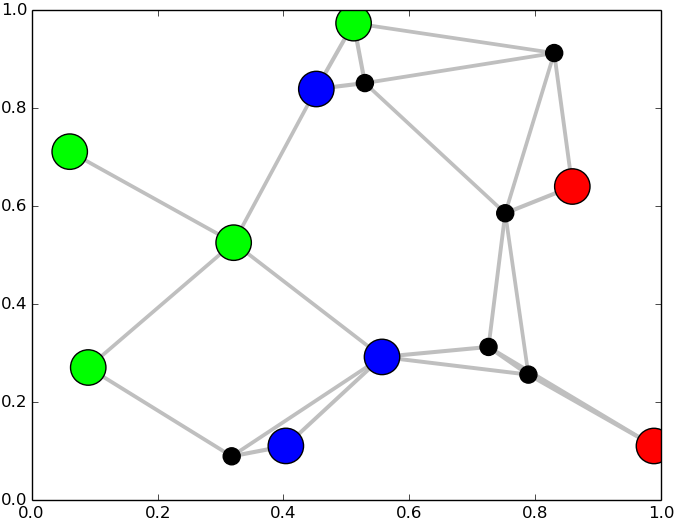}}
    \hfill
    \subfloat[The system at time 150.\label{fig:par400}] {
      \includegraphics[width=.485\columnwidth]{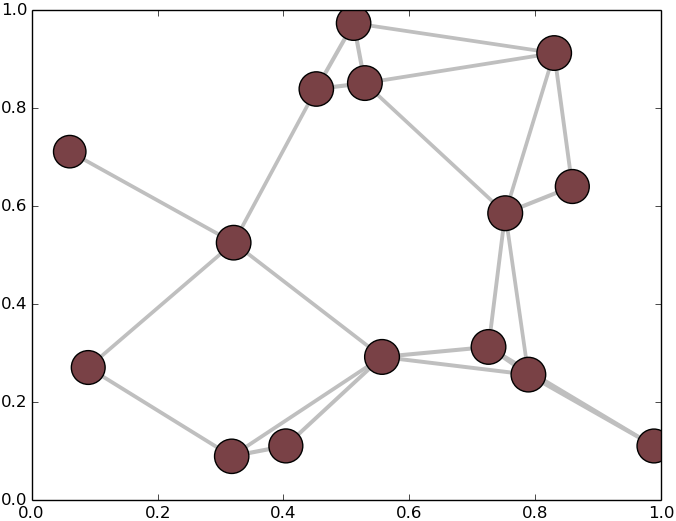}}
    \caption{This three-virus system meets the assumptions of Theorem \ref{parallel1} and the viruses converge to a parallel equilibrium.  
    The colors and diameters follow \eqref{eq:color} and \eqref{eq:diam}. 
A video of this simulation can be found at  \href{https://youtu.be/sy_RoUP7qUs}{youtu.be/sy\_RoUP7qUs}.}
\label{fig:par}
\end{figure}

We conclude with a simulation that implements the control techniques presented in Section \ref{sec:control}. Consider the single virus system in Figure \ref{fig:con}. This system is homogeneous in infection rate, with $\beta = 0.492$. 
We compare the system with no controller (on the left), a controller using Problem \ref{prob:1} (in the middle), and a controller that uses Algorithm \ref{alg:1} to solve Problem \ref{prob:2} iteratively with $\kappa=.05$ (on the right). The sum of the final probabilities of infection for all agents ($\sum_{i=1}^n p_i(100)$) for the three plots are 10.7, 4.92, and 3.6, respectively.  
Therefore, Algorithm \ref{alg:1} performed the best, however both had significant improvements over the uncontrolled simulation. The maximum eigenvalues of the three linearized systems are, from left to right, 1.893, 0.557, and 0.421; so none of the linearized systems are Hurwitz. Therefore, consistent with Theorem \ref{thm3}, the systems are all at NDFE. However, even though the control efforts do not completely eradicate the virus, they do mitigate its effect.

\begin{figure}
    \centering
    \subfloat[The system at time zero.\label{fig:nocon}] {
      \includegraphics[width=\columnwidth]{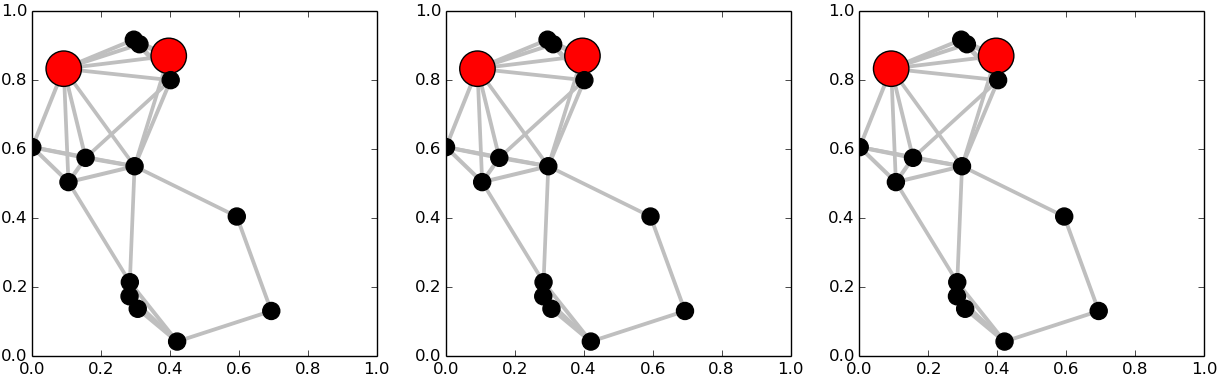}}
    \hfill
    \subfloat[Final state with no control, Problem \ref{prob:1}, and implementing Algorithm \ref{alg:1} on Problem \ref{prob:2}.\label{fig:conS}] {
      \includegraphics[width=\columnwidth]{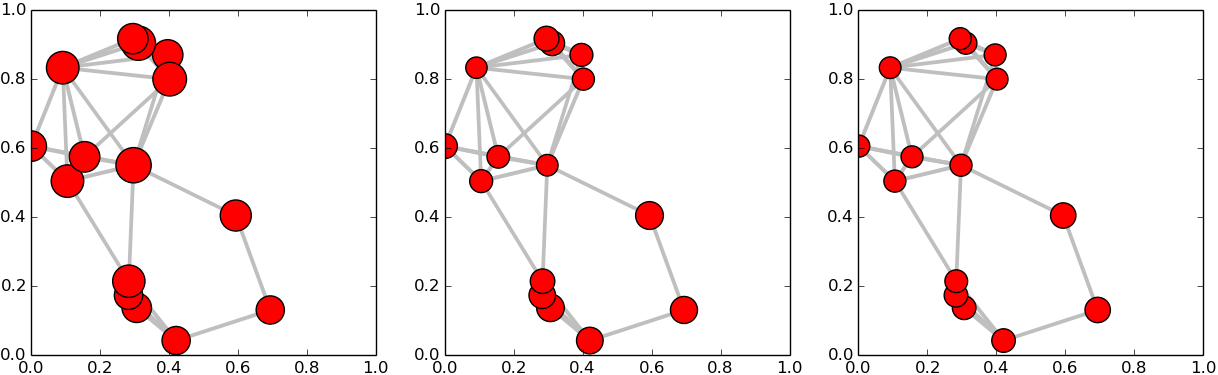}}
    \caption{This is a single virus epidemic equilibrium comparing control techniques. A video of this simulation can be found at  \href{https://youtu.be/P0k5VYxUFJ0}{youtu.be/P0k5VYxUFJ0}.
    }
\label{fig:con}
\end{figure}

\section{Conclusion}\label{sec:con}

We have explored the competing multi-virus SIS model with several theorems exploring stability of the equilibria of the model for the static and time--varying graph cases. We have also proposed several control techniques that appeal to Theorems \ref{thm:1}, \ref{thm:1t}, and \ref{thm:nonH}, providing two efficient centralized antidote distribution/allocation protocols.

In future work we would like to explore more generic cases of co-existing epidemic states. Further, we would like to compare the techniques in Section \ref{sec:control} to other existing techniques.  We would also like to implement the control techniques on large scale systems with at least tens of thousands of nodes.

\bibliographystyle{IEEEtran}
\bibliography{IEEEabrv,bib}



\end{document}